\documentclass[12pt, reqno]{amsart}
\usepackage{amsmath, amsthm, amscd, amsfonts, amssymb, graphicx, color}
\usepackage[bookmarksnumbered, colorlinks, plainpages]{hyperref}
\usepackage{lscape}
\hypersetup{colorlinks=true,linkcolor=red, anchorcolor=green, citecolor=cyan,
urlcolor=red, filecolor=magenta, pdftoolbar=true}

\newcommand{\hide}[1]{} 

\newcommand{\tblue}[1]{\textcolor{blue}{#1}}

\catcode`\@=11
\def\omathop#1#2#3{\let\temp=#1\def\letter{#2} \ifcat#3_
\let\next\@@olim\else\let\next\@olim\fi\next#3}
\def\@olim{\letter\text{-}\!\temp}
\def\@@olim_#1{\mathchoice{
 \setbox0=\hbox{$\displaystyle\letter\text{-}\!\temp\!\text{-}\letter$}
 \setbox2=\hbox{$\displaystyle\temp$}
 \setbox4=\hbox{$\scriptstyle#1$}
 \dimen@=\wd4 \advance\dimen@ by -\wd2 \divide\dimen@ by2
 \def\next{\letter\text{-}\!\temp_{\hbox to 0pt{\hss$\scriptstyle#1$\hss}}
 \hskip\dimen@}
 \ifdim\wd2&gt;\wd4 \def\next{\@olim_{#1}}\fi
 \ifdim\wd4&gt;\wd0 \def\next{\mathop{\llap{$\letter$-}\!\temp}\limits_{#1}}\fi
 \next}
 {\@olim_{#1}}{\@olim_{#1}}{\@olim_{#1}}}

\def\olim{\omathop{\lim}{(o)}}

\newcommand{\nl}{\left\|} 
\newcommand{\nr}{\right\|} 
\newcommand{\bl}{\left|} 
\newcommand{\br}{\right|} 
\newcommand{\be}{\begin{equation}}
\newcommand{\ee}{\end{equation}}
\newcommand{\sideremark}[1]{\setlength{\marginparwidth}{2.7cm}}

\def\eps{\varepsilon}
\def\la{\lambda}
\def\phi{\varphi}

\newcommand{\bi}{\begin{itemize}}
\newcommand{\ei}{\end{itemize}}
\newcommand{\bn}{\begin{enumerate}}
\newcommand{\en}{\end{enumerate}}
\def\R{\mathbb{R}}
\def\N{\mathbb{N}}

\def\cL{\mathcal{L}}

\def\eins{\mathbf{1}}

\newtheorem{thm}{Theorem}[section]

\newtheorem{prop}[thm]{Proposition}

\newtheorem{lemma}[thm]{Lemma}

\newtheorem{example}[thm]{Example}
\theoremstyle{definition}
\newtheorem{definition}[thm]{Definition}

\numberwithin{equation}{section}

\begin{document}

\title[On $C$-compact operators]{On $C$-compact orthogonally additive operators}

\author{Marat~Pliev\MakeLowercase{ 
and} Martin~R.~Weber}

\address{Southern Mathematical Institute of the Russian Academy of Sciences\\
Vladikavkaz, 362027 Russia}

\email{maratpliev@gmail.com}


\address{Technical University Dresden \\
Department of Mathematics, Institute of Analysis \\
01062 Dresden, Germany}

\email{martin.weber@tu-dresden.de}



\begin{abstract}
We consider $C$-compact orthogonally additive operators in vector
lattices. After providing some examples of $C$-compact orthogonally additive operators 
on a vector lattice with values in a Banach space we show that the set of those 
operators is a projection band
in the Dedekind complete vector lattice of all regular orthogonally additive operators.
In the second part of the article we introduce a
new class of vector lattices, called $C$-complete, and 
show that any laterally-to-norm continuous $C$-compact
orthogonally additive operator from a $C$-complete vector lattice
to a Banach space is narrow, which generalizes a result of
Pliev and Popov. 
 \end{abstract}

\maketitle

\section{Introduction}

Orthogonally  additive operators in vector lattices first were
investigated in \cite{Maz-1}. Later these results were extended in
\cite{AP,AP-1,F-1,F-2,F-3,PF,PW,PPW}). Recently,
some connections with problems of the convex geometry  were revealed
\cite{TrVi1,TrVi2}. Orthogonally additive operators in
lattice-normed spaces were studied in \cite{AP-2}. In this paper we
continue this line of research. We analyze the notion of $C$-compact
orthogonally additive operator. In the first part of the article we
show that set of all $C$-compact orthogonally additive operators
from a vector lattice $E$ to an order continuous
Banach lattice $F$ is a projection band in the vector lattice of all
regular orthogonally additive operators from $E$ to $F$ (Theorem
\ref{thm-02}). 
In the final part of the paper  we introduce a new class of
vector lattices which we call $C$-complete (the precise definition is
given in section~$4$) and prove that any laterally-to-norm continuous
$C$-compact orthogonally additive operator from an atomless $C$-complete vector
lattice $E$ to a Banach space $X$ is narrow (Theorem \ref{thm-2}).
This is a generalization of the result of the article \cite[Theorem~3.2]{PP}. 

Note that linear narrow operators in function spaces first appeared in \cite{PPlich}. 
Nowadays the theory of narrow operators is a well-studied object of functional analysis and is presented 
in many research articles (\cite{FFKP,MMP,MP,M,PP,PPW}) and in the monograph \cite{PRan}.

\section{Preliminaries}
All necessary information on vector lattices one can find in \cite{AB}. 
In this article all vector lattices are assumed to be Archimedean.

Two elements $x,y$ of a vector lattice $E$ are called 
\textit{disjoint} (written as $x\perp y$), if $|x|\wedge|y|=0$. 
An element $a>0$ of $E$ is an \textit{atom} if $0\leq x\leq a, \;0\leq
y\leq a$ and $x\bot y$ imply that either $x=0$ or $y=0$. 
The equality $x=\bigsqcup\limits_{i=1}^{n}x_{i}$ means that
$x=\sum\limits_{i=1}^{n}x_{i}$ and  $x_{i}\bot x_{j}$ for all $i\neq
j$. In the case of $n=2$ we write $x = x_1\sqcup x_2$.
An element $y$ of a  vector lattice $E$ is called a \textit{fragment} 
 (or a \textit{component}) of $x \in E$, if
$y\perp(x-y)$. The notation $y \sqsubseteq x$ means that $y$ is a
fragment of $x$.  Two fragments $x_{1}$ and $x_{2}$ of an element
$x$ are said to be \textit{mutually complemented}  if $x = x_1
\sqcup x_2$. The set of all fragments of an element  $x\in E$ is
denoted by $\mathcal{F}_{x}$.

\begin{definition} \label{def:ddmjf0}
Let be $E$ a vector lattice and $X$ a real vector space. An
operator $T:E\rightarrow X$ is said to be \textit{orthogonally additive}
if $T(x+y)=Tx+Ty$ for any  disjoint elements $x,y\in E$.
\end{definition}

It is not hard to check that $T(0)=0$. The set of all orthogonally additive
operators from $E$ to $X$ is a real vector space with respect to the natural
linear operations.

\begin{definition}
Let $E$ and $F$ be vector lattices. An orthogonally additive (in general,
nonlinear) operator $T:E\rightarrow F$ is said to be:
\begin{itemize}
  \item \textit{positive} if $Tx \geq 0$ for all $x \in E$,
  \item \textit{regular},  if  $T=S_{1}-S_{2}$ for two positive, orthogonally
additive operators $S_{1}$ and $S_{2}$ from $E$ to $F$.
\end{itemize}
\end{definition}
The sets of all positive, regular 
orthogonally additive operators from $E$ to $F$ are denoted by
$\mathcal{OA}_{+}(E,F)$, $\mathcal{OA}_{r}(E,F)$, 
respectively, where 
the order in $\mathcal{OA}_r(E,F)$ is introduced as follows:  $S\leq T$ whenever
$(T-S)\geq 0$.
Then $\mathcal{OA}_r(E,F)$ becomes an ordered vector space. 

For a Dedekind complete
vector lattice $F$ we have the following  property of $\mathcal{OA}_{r}(E,F)$.
\begin{prop}[{\cite[Theorem~3.6]{PR}}] \label{thm:PK}
Let $E$ and $F$ be a vector lattices, and assume that $F$ is Dedekind
complete.Then
$\mathcal{OA}_{r}(E,F)$ is a Dedekind complete vector lattice. 
Moreover, for every $S$, $T\in
\mathcal{OA}_{r}(E,F)$ and every $x\in E$ the following formulas\footnote{\, In
the literature these formulas are known as
the Riesz-Kantorovich formulas.}
hold
\begin{enumerate}
\item~$(T\vee S)(x)=\sup\{Ty+Sz:\,x=y\sqcup z\}$;
\item~$(T\wedge S)(x)=\inf\{Ty+Sz:\,x=y\sqcup z\}$;
\item~$T^{+}(x)=\sup\{Ty:\,y\sqsubseteq x\}$;
\item~$T^{-}(x)=-\inf\{Ty:\,\,\,y\sqsubseteq x\}$;
\item~$|Tx|\leq|T|(x)$.
\end{enumerate}
\end{prop}

\section{The projection band of $C$-compact orthogonally additive operators}

In this section we 
show that the set of all
$C$-compact regular orthogonally additive operators from a vector
lattice $E$ to a Banach lattice $F$ with order continuous norm is a
band in the vector lattice of all orthogonally additive regular
operators from $E$ to $F$.

Consider some examples.
\begin{example}\label{Ur}
Assume that $(A,\Xi,\mu)$ and $(B,\Sigma,\nu)$ are $\sigma$-finite  measure spaces. %
We say that a map  $K:A\times
B\times\mathbb{R}\rightarrow\mathbb{R}$
is a {\it Carath\'{e}odory function} if there hold the conditions:
\begin{enumerate}
  \item $K(\cdot,\cdot,r)$ is $\mu\times\nu$-measurable for all $r\in\mathbb{R}$;
  \item $K(s,t,\cdot)$ is continuous on $\mathbb{R}$ for $\mu\times\nu$-almost all $(s,t)\in A\times B$.
\end{enumerate}
If additionally there holds $K(s,t,\cdot)=0$ for $\mu\times\nu$-almost all $(s,t)\in A\times B$ 
then with the function $K$ a nonlinear integral operator: 
$T:\text{Dom}(K)\rightarrow L_{0}(\mu)$  
$$
(Tf)(s):=\int_{B}K(s,t,f(t))\,d\nu(t)\quad\mbox{for}  \: \mu\mbox{-almost \; all} \; s\in A 
$$
is associated, where 
$$
\text{Dom}(K):=\{f\in L_{0}(\nu):\,\int_{B}|K(s,t,f(t))|\,d\nu(t)\in
L_{0}(\mu)\}.
$$

 Assume that $E$ is a Banach function space 
in $L_{0}(\nu)$ and $E\subseteq \text{Dom}(K)$.
Then $T$ defines an \textit{orthogonally additive
 integral operator} acting from $E$ to $L_0(\mu)$.
 The operator $T$ is called 
{\it Urysohn (integral) operator and, the function $K$ is called the
{\it kernel} of this  operator}. If $L_{0}(\mu)$ is replaced by $\Bbb{R}$
we say that $T$ is an \textit{Urysohn integral functional.}
\end{example}

\begin{example}
Assume that $(\Omega,\Sigma,\mu)$ is a measure space.
Then the functional  $\mathcal{N}:L_{1}(\mu)\to \R$ defined by
$$
\mathcal{N}(f)=\| f\|_{L_{1}(\mu)}:=\int_{\Omega}|f(t)|\,d\mu,\quad f\in
L_{1}(\mu)
$$
is positive and orthogonally additive.
\end{example}

Let $E$ be a vector lattice. 
A linear operator $S:E\to E$ is called {\it band preserving} if
$Sx\in\{x\}^{\perp\perp}$ for all $x\in E$. 
Clearly, every band preserving operators preserves disjointness.

\begin{example}
Let $(\Omega,\Sigma,\mu)$ be a measure space and $S:L_{0}(\mu)\to
L_{0}(\mu)$ be a band preserving linear operator. Consider the new
operator $T:L_{0}(\mu)\to L_{0}(\mu)$ defined by
\begin{align*}
Tf=fS(f),\quad f\in L_{0}(\mu).
\end{align*}
Observe that $T$ can be treated as the multiplication operator by the "variable" function. 
It is not hard to verify that  $T$ is an
orthogonally additive operator. 
Indeed, due to 
$Sf\in\{f\}^{\perp\perp}$, for every disjoint $f_{1},f_{2}\in L_{0}(\mu)$ we
have
\begin{align*}
T(f_1+f_2)=(f_1+f_2)S(f_1+f_2)= \\ 
f_1S(f_1+f_2)+f_2S(f_1+f_2)= \\
f_1Sf_1+f_2Sf_2=Tf_1+Tf_2.
\end{align*}
\end{example}
\begin{definition}\label{Nem}
Let  $(\Omega,\Sigma,\mu)$ be a $\sigma$-finite  measure space. We say
that a function $N:\Omega\times \R\to \R$ belongs  to the class
$\mathfrak{S}$ (or $N$ is a $\mathfrak{S}$-function) if
the  following conditions hold:
\begin{enumerate}
  \item \; $N(t,0)=0$ for $\mu$-almost all $t\in \Omega$;
  \item \; $N(\cdot,g(\cdot))$ is $\mu$-measurable for any $g\in
L_{0}(\mu)$.
\end{enumerate}
If a function $N$  satisfies only the condition $(2)$, then it is called a
\textit{ superpositionally measurable function} or
\textit{sup-measurable function}.
\end{definition}

\begin{example}
Let $N:\Omega\times \R\to \R$ be a $\mathfrak{S}$-function. Then with $N$
there is associated an orthogonally additive operator
$T_{N}:L_{0}(\mu)\to L_{0}(\mu)$ defined  by
$$
T_{N}(g)(t)=N(t,g(t)) \;\,\mbox{for $\mu$-almost all } \, t\in \Omega \;\,
\mbox{and } \;g\in L_{0}(\mu).
$$
\end{example}
\hide{
For convenience of the reader we present a short proof.
\begin{proof}
For $f\in L_{0}(\mu)$ we show that
\[  N(t,f\eins_D(t))=N(t,f(t))\eins_D(t) \quad \mbox{for all} \; D\in \Sigma,    \]
where as usual, $\eins_D$ denotes the characteristic function of the set $D$.
Indeed,  take $t\in D$. Then
\[  N(t, f\eins_D(t)) = N(t, f(t)) = N(t,f(t))\eins_D(t).   \]
Assume that  $t\notin D$. Then by condition $(C_1)$ of the
Definition~\ref{Nem} we have
\[  N(t, f\eins_D(t)) = N(t, 0)= 0 = N(t,f(t))\eins_D(t).   \]
We recall that the {\it support} of $f\in L_{0}(\mu)$
is the measurable set $\text{supp}\,f:=\{t\in A:\,f(t)\neq 0\}$.
Take $f,g\in L_{0}(\mu)$ with $f\perp g$. Then there exist
measurable sets $D_1$ and $D_2$, such that $D_1\cap D_2=\emptyset$,
$D_1=\text{supp}(f)$ and $D_2=\text{supp}(g)$.
  Thus for almost all $t\in A$ we have
\begin{gather*}
T_N(f+g)(t)=N(t,(f+g)(t))\eins_{D_1\sqcup D_2}(t)= \\
N(t,(f+g)(t))\eins_{D_1}(t)+N(t,(f+g)(t))\eins_{D_2}(t) = \\
N(t,f\eins_{D_1}(t))+N(t,g\eins_{D_2}(t))=\\
N(t,f(t))+N(t,g(t))=T_N(f)(t)+T_N(g)(t)
\end{gather*}
and the  orthogonal additivity of  $T_N$ is proved.
\end{proof}
 }

This class of operators (sometimes called {\it nonlinear superposition
operators} or {\it Nemytskii operators}) is widely represented in the literature 
(see e.g. \cite{AZ}).

\begin{definition}
Let be $E$ be a vector lattice and $Y$  a normed space.
An orthogonally additive operator $T:E\to Y$ is called 
\begin{itemize}
\item \textit{$AM$-compact} provided $T$ maps 
  order bounded subsets of $E$ into relatively compact subsets in  $Y$,
\item \textit{$C$-compact}, if  $T(\mathcal{F}_{x})$ is
a relatively compact in  $Y$ for any $x\in E$.
\end{itemize}
For a Banach lattice $F$ by $\mathcal{COA}_{r}(E,F)$ is denoted the space of
all
$C$-compact regular orthogonally additive operators from
$E$ to $F$.
\end{definition}

\begin{example}
We note that 
$\mathcal{OA}_r(\R,\R)$ is exactly the set of all real-valued functions such that $f(0)=0$.
Define an orthogonally additive operator $T:\R\to \R$  by
$$
Tx=\begin{cases} \frac{1}{x^2},\;\;\text{if $x\neq
0$}\\
\; 0,\;\;\;\text{if $x=0$.}\\
\end{cases}
$$
Since any element $x\in\R$ is an atom  one has $\mathcal{F}_x=\{0,x\}$ for any
$x\in\R$. It follows that $T$ is a $C$-compact operator.
On the other hand  $T([0,1])$ is an unbounded set in $\R$ and therefore
$T$ is not $AM$-compact.
\end{example}

\begin{prop}
Let $(B,\Sigma,\nu)$ be $\sigma$-finite  measure spaces $E$ be a Banach function space 
in $L_{0}(\nu)$ and $T\colon E\to\Bbb{R}$ be the Urysohn integral functional defined by
$$
Tf=\int_{B}K(t,f(t))\,d\nu(t),\quad f\in E
$$
with the kernel $K$. Then $T$ is $C$-compact.
\end{prop}
\begin{proof}
Take $f\in E$. 
We note that $\mathcal{F}_{f}$
coincides with the set $\{f\eins_D:\,D\in\Sigma\}$. Now for every
$D\in\Sigma$ we write 
\begin{align*}
Tf\eins_D=\int_{B}K(t,f\eins_{D}(t))\,d\nu(t)=\int_{D}K(t,f(t))\,d\nu(t)\leq\\
\int_{B}|K(t,f(t))|\,d\nu(t)=M.
\end{align*}

Hence the set $T(\mathcal{F}_{f})$ is order bounded in $\Bbb{R}$ and
therefore the operator $T$ is $C$-compact.
\end{proof}

We mention  that a $C$-compact order bounded orthogonally additive  operator 
$T:E \to F$ from a Banach
lattices $E$ to a $\sigma$-Dedekind complete Banach lattice $F$ is $AM$-compact if, in
addition, $T$ is uniformly continuous
on order bounded subsets of $E$ \cite[Theorem~3.4]{Maz-2}.

The norm in a normed vector lattice is {\textit order continuous} if
$x_\alpha \downarrow 0$ implies $\|x_\alpha\|\downarrow 0$. We point
out that a Banach lattice with order continuous norm is Dedekind
complete (see \cite[Theorem 12.9]{AB}). 

Now we are ready to present the first main result of the article. 
\begin{thm} \label{thm-02}
Let be $E$ a vector lattice and $F$ a Banach lattice with order
continuous norm. Then
 the set  of all $C$-compact regular orthogonally additive operators from $E$
to
$F$ is a projection band in $\mathcal{OA}_{r}(E,F)$.
\end{thm}

In order to prove Theorem~\ref{thm-02} we  need some auxiliary propositions.
\begin{prop} \cite[Proposition~3.9]{OPR}\label{OPR}
Let $E$ be a vector lattice and $x,y\in E$. Then $x\sqsubseteq y$ if and only
if
$x^{+}\sqsubseteq y^{+}$ and $x^{-}\sqsubseteq y^{-}$.
\end{prop}\label{le:02}
\begin{prop}\label{Riesz}
Let $E$ be a vector lattice and
$\bigsqcup\limits_{i=1}^{n}x_{i}=\bigsqcup\limits_{k=1}^{m}y_{k}$
for some  $(x_{i})_{i=1}^{n}$ and   $(y_{k})_{k=1}^{m}\subset E.$
Then there exist a family of pairwise disjoint elements $(z_{ik})\subset E$,
where $i\in\{1,\ldots,n\}$ and $k\in\{1,\ldots,m\}$ such that
\begin{enumerate}\label{R-1}
\item[$(i)$] $x_{i}=\bigsqcup\limits_{k=1}^{m}z_{ik}$ for any
$i\in\{1,\ldots,n\}$;
\item[$(ii)$] $y_k=\bigsqcup\limits_{i=1}^{n}z_{ik}$ for  any
$k\in\{1,\ldots,m\}$;
\item[$(iii)$]
$\bigsqcup\limits_{i=1}^{n}\bigsqcup\limits_{k=1}^{m}z_{ik}=\bigsqcup\limits_{
i=1}^{n}x_{i}=\bigsqcup\limits_{k=1}^{m}y_{k}$.
\end{enumerate}
\end{prop}
\begin{proof}
Using  Proposition~\ref{OPR}, and applying induction arguments we have
\begin{gather*}
\bigsqcup\limits_{i=1}^{n}x_i^{+}=\Bigl(\bigsqcup\limits_{i=1}^{n}x_i\Bigl)^{+}
=
\Bigl(\bigsqcup\limits_{k=1}^{m}y_k\Bigl)^{+}=\bigsqcup\limits_{k=1}^{m}y_k^{+}
;\\
\bigsqcup\limits_{i=1}^{n}x_i^{-}=\Bigl(\bigsqcup\limits_{i=1}^{n}x_i\Bigl)^{-}
=
\Bigl(\bigsqcup\limits_{k=1}^{m}y_k\Bigl)^{-}=\bigsqcup\limits_{k=1}^{m}y_k^{-}.
\end{gather*}
Now, by the Riesz Decomposition property 
there exist
$$
z_{ik}^{+},z_{ik}^{-}\in E_{+ }, \quad i\in\{1,\ldots,n\}\; \mbox{ and }
\;k\in\{1,\ldots,m\}
$$
such that
\begin{gather*}
x_{i}^{+}=\bigsqcup\limits_{k=1}^{m}z_{ik}^{+},  \quad
x_i^{-}=\bigsqcup\limits_{k=1}^{m}z_{ik}^{-}, \;\, i\in\{1,\ldots,n\}, \\
y_{k}^{+}=\bigsqcup\limits_{i=1}^{n}z_{ik}^{+},  \quad
y_k^{-}=\bigsqcup\limits_{i=1}^{n}z_{ik}^{-}, \;\,k\in\{1,\ldots,m\}.
\end{gather*}
Set $z_{ik}=z_{ik}^{+}-z_{ik}^{-}, \quad   i\in\{1,\ldots,n\} \mbox{ and }
\;k\in\{1,\ldots,m\}$.
Hence
\begin{gather*}
\bigsqcup\limits_{i=1}^{n}\bigsqcup\limits_{k=1}^{m}z_{ik}=\bigsqcup\limits_{i=1
}^{n}x_{i}=\bigsqcup\limits_{k=1}^{m}y_{k}
\end{gather*}
and  the proof is completed.
\end{proof}

\begin{lemma}\label{lemma-direct}
Let $E$ and $F$ be  vector lattices with $F$  Dedekind complete,
$S,T\in\mathcal{OA}_{r}(E,F)$ and $x\in E$. Then the following equalities hold:
\begin{enumerate}\label{formulas}
\item~$(T\vee S)x=\sup\Big\{\sum\limits_{i=1}^{n}Tx_i\vee
Sx_i:\,x=\bigsqcup\limits_{i=1}^{n}x_i,\,n\in\N\Big\}$;
\item~$(T\wedge S)x=\inf\Big\{\sum\limits_{i=1}^{n}Tx_i\wedge
Sx_i:\,x=\bigsqcup\limits_{i=1}^{n}x_i,\,n\in\N\Big\}$;
\item~$|T|x=\sup\Big\{\sum\limits_{i=1}^{n}|Tx_i|:\,x=\bigsqcup\limits_{i=1}^{n}
x_i,\,n\in\N\Big\}$.
\end{enumerate}
\end{lemma}

\begin{proof}
Since $x\wedge y=-(-x)\vee(-y)$ and $|x|=x\vee(-x)$  for proving the
lemma it is sufficient to establish only the equation $(1)$. Put
$$
\mathfrak{A}(x):=\Big\{\sum\limits_{i=1}^{n}Tx_i\vee
Sx_i:\,x=\bigsqcup\limits_{i=1}^{n}x_i,\,n\in\N\Big\}.
$$
We show that $\mathfrak{A}(x)$ is an upward directed set. Indeed, take two
disjoint decompositions
$x=\bigsqcup\limits_{i=1}^{n}x_i$
and $x=\bigsqcup\limits_{k=1}^{m}y_k$ of $x$. By Proposition~\ref{Riesz} there
exists a disjoint
decomposition $x=\bigsqcup\limits_{i=1}^{n}\bigsqcup\limits_{k=1}^{m}z_{ik}$
such that
\begin{gather*}
x_{i}=\bigsqcup\limits_{k=1}^{m}z_{ik}, \;\, i\in\{1,\ldots,n\},\\
y_k=\bigsqcup\limits_{i=1}^{n}z_{ik}, \;\, k\in\{1,\ldots,m\}.
\end{gather*}
We observe that
$$
Sx_{i}\vee Tx_{i}\leq \sum\limits_{k=1}^{m}Tz_{ik}\vee Sz_{ik}
$$
for any $i\in\{1,\ldots,n\}$. Indeed
\begin{gather*}
Sx_i=S\Big(\bigsqcup\limits_{k=1}^{m}z_{ik}\Big)=\sum\limits_{k=1}^{m}Sz_{ik}
\leq \sum\limits_{k=1}^{m}Tz_{ik}\vee Sz_{ik}, \\
Tx_i=T\Big(\bigsqcup\limits_{k=1}^{m}z_{ik}\Big)=\sum\limits_{k=1}^{m}Tz_{ik}
\leq \sum\limits_{k=1}^{m}Tz_{ik}\vee Sz_{ik}
 \Longrightarrow\\
Sx_{i}\vee Tx_{i}\leq \sum\limits_{k=1}^{m}Tz_{ik}\vee Sz_{ik}.
\end{gather*}
Similar arguments show that
$$
Sy_{k}\vee Ty_{k}\leq \sum\limits_{i=1}^{n}Tz_{ik}\vee Sz_{ik}
$$
for any $k\in\{1,\ldots,m\}$.
Now we may write
\begin{gather*}
\sum\limits_{i=1}^{n}Tx_i\vee
Sx_i\leq\sum\limits_{i=1}^{n}\sum\limits_{k=1}^{m}Tz_{ik}\vee
Sz_{ik}\quad \mbox{and} \quad \sum\limits_{k=1}^{m}Ty_k\vee
Sy_k\leq\sum\limits_{i=1}^{n}\sum\limits_{k=1}^{m}Tz_{ik}\vee
Sz_{ik},
\end{gather*}
and deduce that $\mathfrak{A}(x)$ is an upward directed subset of
$F$. By Proposition~\ref{thm:PK} $\mathcal{OA}_{r}(E,F)$ is a
Dedekind complete vector lattice and 
\[(T\vee S)x=\sup\{Ty+Sz:\,x=y\sqcup z\}. \] 
Clearly $Tx\vee Sx\leq (T\vee S)x$
for any $x\in E$. Thus for any  decomposition
$x=\bigsqcup\limits_{i=1}^{n}x_i, n\in \N$ we have that
\begin{gather*}
\sum\limits_{i=1}^{n}Tx_i\vee Sx_i\leq\sum\limits_{i=1}^{n}(T\vee S)x_i=(T\vee
S)x
\end{gather*}
and therefore the set $\mathfrak{A}(x)$ is upper order bounded. Put
$f_x:=\sup\mathfrak{A}(x)$. Then $f_x\leq(T\vee S)x$. On the other
hand for any disjoint decomposition $x=y\sqcup z$ we have
\begin{gather*}
Ty+Sz\leq \sup\Big\{Ty\vee Sy+Tz\vee Sz:\,x=y\sqcup z\Big\}\leq\\
\sup\Big\{\sum\limits_{i=1}^{n}Tx_i\vee
Sx_i:\,x=\bigsqcup\limits_{i=1}^{n}x_i,\,n\in\N\Big\}=f_x.
\end{gather*}
Passing to the supremum 
in the last inequality over all disjoint decompositions $x=y\sqcup z$ of $x$, due to
formula $(1)$ of Proposition \ref{thm:PK}, we deduce $(T\vee
S)x\leq f_x$. It follows $(T\vee S)x=f_x$ which completes the proof.
\end{proof}
Below we shall use the following elementary observation. Since
\[ \Big\{\sum\limits_{i=1}^{n}Sx_i\vee Tx_i:\,x=\bigsqcup\limits_{i=1}^{n}x_i,\,n\in\N\Big\}\] 
is an upward
downward directed subset of $F$ the set
\[ \Big\{\sum\limits_{i=1}^{n}Sx_i\wedge Tx_i:\,x=\bigsqcup\limits_{i=1}^{n}x_i,\,n\in\N\Big\}\] 
is  downward
directed.

\begin{lemma}\label{lemma-frag}
Let be $E$ a vector lattice,  $F$ a Banach lattice with order
continuous norm and $T\in\mathcal{COA}_{r}(E,F)$. Then
$\mathcal{F}_{T}\subset\mathcal{COA}_{r}(E,F)$.
\end{lemma}
\begin{proof}
We recall that by definition $S\in\mathcal{F}_{T}$ if $|S|\wedge
|T-S|=0$. We show that $S(\mathcal{F}_{x})$ is a relatively compact
set\footnote{\, Due to the norm-completeness of $F$ this is
equivalent to its totally boundedness.}  in $F$ for any
$S\in\mathcal{F}_{T}$ and $x\in E$. Indeed, fix $x\in E$ and
$\varepsilon > 0$. Since
$\Big\{\sum\limits_{i=1}^{n}|S|x_i\wedge
|T-S|x_i:\,x=\bigsqcup\limits_{i=1}^{n}x_i,\,n\in\N\Big\}$ is a
downward directed subset of $F$ by Lemma~\ref{lemma-direct} and the
order continuity of the norm in the Banach lattice $F$  there exists
a disjoint decomposition $x=\bigsqcup\limits_{i=1}^{n}x_i$ of $x$
such that
\[
\Big\|\sum\limits_{i=1}^{n}|S|x_i\wedge |T-S|x_i\Big\|
<\frac{\varepsilon}{4}.
\]
Since $T(\mathcal{F}_{x_i})$ is a relatively compact set \sideremark{\tblue{$x_i$}}
for any $i \in \{1,\ldots, n\}$ there exists a finite subset
$D_i\subset \mathcal{F}_{x_i}$ such that for any $w\in
\mathcal{F}_{x_i}$ there is $u\in D_i$ satisfying
$$
\|Tw-Tu\|<\frac{\varepsilon}{2n}.
$$
Moreover, by Proposition~\ref{Riesz} for any $y\in\mathcal{F}_{x}$
there exists a disjoint decomposition
$y=\bigsqcup\limits_{i=1}^{n}y_i$, where $y_i\sqsubseteq x_i$, i.e.
$y_i\in\mathcal{F}_{x_i}$. Hence, for some $u_i\in D_i$ the
following inequality
$$
\|Ty_i-Tu_i\| < \frac{\varepsilon}{2n}
$$
holds. We remark that for any positive orthogonally additive
operator $G:E\to F$ one has\footnote{\;This immediately follows from
$\bl y\br \wedge \bl x-y\br=0$, the positivity of $G$ and
$G(x)=G(y)+G(x-y)\geq G(y)$.} $Gw\leq Gx_i$ for all
$w\in\mathcal{F}_{x_i}$. Now, by taking into account the inequality
$|x|\leq |x+y|+|x|\wedge |y|$, which holds in any vector
lattice\footnote{\;Indeed,  since (see \cite{Zaa}, Theorems 5.1 and
5.5) $ |x+y|\geq||x|-|y||\geq(|x|-|y|)^+$ we have
\[
|x|-|x|\wedge|y|=|x|+(-|x|)\vee(-|y|)=0\vee(|x|-|y|)=(|x|-|y|)^+\leq|x+y|.
\]}
we have
\[
 \begin{array}{lll}
 |Sy_i-Su_i| \leq  &  &  \\
|Sy_i-Su_i+ (T-S)y_i-(T-S)u_i|  +  |Sy_i-Su_i|\wedge|(T-S)y_i-(T-S)u_i|  \leq
&   &   \\
|Ty_i-Tu_i|+(|Sy_i|+|Su_i|)\wedge(|(T-S)y_i|+|(T-S)u_i|)                 \leq
&   &  \\
|Ty_i-Tu_i|+(|S|y_i+|S|u_i)\wedge(|T-S|y_i+|T-S|u_i)                      \leq
&   &    \\
|Ty_i-Tu_i|+(|S|x_i+|S|x_i)\wedge(|T-S|x_i+|T-S|x_i)                        =
&   &     \\
|Ty_i-Tu_i|+ 2\Big(|S|x_i\wedge|T-S|x_i\Big). &   &
 \end{array}
\]
Then we may write
\begin{align*}
\sum\limits_{i=1}^{n}|Sy_i-Su_i|\leq \sum\limits_{i=1}^{n}|Ty_i-Tu_i|+
 2\sum\limits_{i=1}^{n}\Big(|S|x_i\wedge|T-S|x_i\Big).
\end{align*}
Put $D:=\{u=\bigsqcup\limits_{i=1}^{n}u_i:\,\,\,u_i\in D_i\}.$
Clearly $D$ is a finite subset of $\mathcal{F}_{x}$.
Thus for any $y\in\mathcal{F}_{x}$ and $u\in D$ we have
\begin{align*}
\|Sy-Su\|=
\Big\|\big|S\Big(\bigsqcup\limits_{i=1}^{n}y_i\Big)-S\Big(\bigsqcup\limits_{i=1}
^ nu_i\Big)\big|\Big\|  =
\Big\|\big|\sum\limits_{i=1}^{n}Sy_i-Su_i\big|\Big\|\leq  \\
\Big\|\sum\limits_{i=1}^{n}|Sy_i-Su_i|\Big\|\leq
\Big\|\sum\limits_{i=1}^{n}|Ty_i-Tu_i|\Big\|+
 2\Big\|\sum\limits_{i=1}^{n}\Big(|S|x_i\wedge|T-S|x_i\Big)\Big\|\leq  \\
\sum\limits_{i=1}^{n}\|Ty_i-Tu_i\|+\frac{\varepsilon}{2} \leq \varepsilon.
\hspace{8.3cm}
\end{align*}
Hence $\{Su\colon u\in D\}$ is a finite $\eps$-net for $S(\mathcal{F}_x)$,
and therefore the proof is completed.
\end{proof}

Let $V$ be a vector lattice and $v\in V_{+}$. The order ideal
generated by $v$ is denoted by $I_{v}$, i.e. $I_v=\{x\in V\colon
\exists \alpha>0 \;\mbox{ such that }\;|x|\leq \alpha |v| \}$. A
{\it $v$-step function} is any vector $s\in V$ for which there exist
pairwise disjoint fragments $v_{1},\dots,v_{n}$ of $v$ with
$v=\bigsqcup\limits_{i=1} v_{i}^n$ and real numbers
$\lambda_{1},\dots,\lambda_{n}$ satisfying
$s=\sum\limits_{i=1}^{n}\lambda_{i}v_{i}$. 
The next proposition is known as the Freudenthal Spectral Theorem (\cite{AB}, Theorem~2.8).

\begin{prop}\label{Fr}
Let $V$ be a vector lattice with the principal projection property
and let $v\in V_{+}$. Then for every $u\in I_v$ there exists a
sequence $(s_n)_{n\in \N}$ of $v$-step functions satisfying $\,0\leq
u-s_{n}\leq\frac{1}{n}v\,$  for each $n$ and $s_{n}\uparrow u$.
Moreover, $0\leq s_n$ if $\,0\leq u$.
\end{prop}
Now we ready to prove the first main result.
\begin{proof}[Proof of Theorem~\ref{thm-02}]
We prove some properties of $\mathcal{COA}_{r}(E,F)$:  \\
a) Clearly, $\mathcal{COA}_{r}(E,F)$ is a vector subspace of
$\mathcal{OA}_{r}(E,F)$. \\
b) We show that $\mathcal{COA}_{r}(E,F)$ is even a vector sublattice
of $\mathcal{OA}_{r}(E,F)$. Take $S,\, T\in\mathcal{COA}_{r}(E,F)$.
Then $T-S\in \mathcal{COA}_{r}(E,F)$. By Lemma~\ref{lemma-frag}
$\mathcal{F}_{T}\subset\mathcal{COA}_{r}(E,F)$ and therefore
$T^+\in\mathcal{COA}_{r}(E,F)$. Therefore due to the equalities
\begin{align*}
S+(T-S)^+=S+(T-S)\vee 0=T\vee S, \quad
S\wedge T=-(-S)\vee(-T)
\end{align*}
which are valid in $\mathcal{OA}_{r}(E,F)$
we obtain that $\mathcal{COA}_{r}(E,F)$ is a sublattice of
$\mathcal{OA}_{r}(E,F)$. \\
c) Now we show that $T\in\mathcal{COA}_{r}(E,F)$ if $0\leq
T_{\lambda}\uparrow T$ in $\mathcal{OA}_{r}(E,F)$ and any
$T_\lambda\in\mathcal{COA}_{r}(E,F)$. Indeed, take $x\in E$ and
$\varepsilon > 0$. Since the Banach lattice $F$ is order continuous
it follows from $T_{\lambda}x\uparrow Tx$ that
$\|Tx-T_{\lambda_{0}}x\|<\frac{\varepsilon}{4}$ for some
$\lambda_0$. Then actually
$\|Ty-T_{\lambda_{0}}y\|<\frac{\varepsilon}{4}$ for any
$y\in\mathcal{F}_{x}$. Indeed, consider $x=y\sqcup z$ for some $z\in
E$. Then
\[
 0  \leq  Tx-T_{\lambda_{0}}x  =  T(y\sqcup z)-T_{\lambda_{0}}(y\sqcup z)=
Ty-T_{\lambda_{0}}y+Tz-T_{\lambda_{0}}z\geq Ty-T_{\lambda_{0}}y
\]
implies $\|Ty-T_{\lambda_{0}}y\|\leq\|Tx-T_{\lambda_{0}}x\|$. Since
$T_{\lambda_{0}}\in\mathcal{COA}_{r}(E,F)$ there exists a finite
subset $D$ of $\mathcal{F}_{x}$ with the property that for any
$y\in\mathcal{F}_{x}$ there exists $u\in D$ satisfying
\[
\|T_{\lambda_{0}}u-T_{\lambda_{0}}y\|<\frac{\varepsilon}{2}.
\]
So we obtain
\begin{align*}
 \|Tu-Ty\|\leq\|Tu-T_{\lambda_{0}}u+T_{\lambda_{0}}u-Ty+T_{\lambda_{0}}y-T_{
\lambda_{0}}y\|\leq\\
\|Tu-T_{\lambda_{0}}u\|+\|Ty-T_{\lambda_{0}}y\|+\|T_{\lambda_{0}}u-T_{\lambda_{0
}}y\|<\varepsilon,
\end{align*}
what establishes the relative compactness of $T(\mathcal{F}_T)$ in $F$. \\
d) Finally we prove that $\mathcal{COA}_{r}(E,F)$ is an order ideal
in  $\mathcal{OA}_{r}(E,F)$. Let $0\leq R \leq T$, where
$R\in\mathcal{OA}_{r}(E,F)$ and $T\in\mathcal{COA}_{r}(E,F)$. Then
$R\in I_T$ and by Proposition~\ref{Fr} there\footnote{\, This proposition can be 
applied since by Proposition~\ref{thm:PK} the vector lattice $\mathcal{OA}_r(E,F)$ is
Dedekind complete.
Any $T$-step-function $S\in \mathcal{OA}_r(E,F)$ has the form
$S=\sum\limits_{i=1}^m\la_i T_i$, where $T_i$ are disjoint fragments
of $T$ such that $T= \bigsqcup\limits_{i=1}^m T_i$.}
exists a sequence $(S_n)_{n\in \N}$ in $\mathcal{OA}_{r}(E,F)$ of $T$-step-functions 
with $0\leq S_n\uparrow R$. 
Taking into account that\footnote{\, This
follows from the fact that together with $T$ each fragment $T_i$ of
$T$ belongs to $\mathcal{COA}_r(E,F)$.}
$S_n\in\mathcal{COA}_{r}(E,F)$, $n\in\N$ and what has been
established in c) we deduce that $R\in\mathcal{COA}_{r}(E,F)$.
So, $\mathcal{COA}_{r}(E,F)$ is a band in $\mathcal{OA}_{r}(E,F)$. \\
e) Due to the Dedekind completeness of $\mathcal{OA}_{r}(E,F)$ it is
a projection band.
\end{proof}
\section{$C$-compact and narrow orthogonally additive operators}\label{sec3}

In this section we consider a new class of vector lattices,
where the condition of  Dedekind completeness is replaced by a much
weaker property. For laterally-to-norm continuous, 
$C$-compact orthogonally additive operators from a $C$-complete vector lattice $E$ 
to a Banach space $X$ we show their narrowness.

\begin{definition} \label{def:c-compact}
A vector lattice $E$ is said to be { \it
$C$-complete}, if for each  $x\in E_{+}$ any subset $D\subset\mathcal{F}_{x}$
has a supremum and an infimum.
\end{definition}

Clearly, every Dedekind complete vector lattice $E$ is
$C$-complete. The reverse statement, in general, is not true.
\begin{prop}\label{contf}
The vector lattice $E=C[0,1]$ of all continuous functions on the
interval $[0,1]$ is $C$-complete.
\end{prop}
\begin{proof}
Fix $f\in E_{+}$. Since the set
$\mathcal{S}_{f}:=\{t\in[0,1]:\,f(t)>0\}$ is open, by {\cite[Chap.2,
Theorem~5]{KF}} there is the decomposition
$\mathcal{S}_{f}=\bigcup\limits_{i=1}^{\infty}(a_i, b_i)$ where
$(a_i, b_i)\cap (a_j, b_j)=\emptyset$, $i\neq j$. Take
$g\in\mathcal{F}_{f}$.
We claim that for every  $i\in\Bbb{N}$ either
$g(t)=0$ or $g(t)=f(t)$ for any $t\in(a_i, b_i)$. Indeed, denote
by $f_{i}$ and $g_i$ the restrictions of $f$ and $g$ on
the closed interval $[a_i, b_i]$, respectively. It is clear that
$g_{i}\in\mathcal{F}_{f_i}$.  Assume that  there exists a nonzero
fragment $v_{i}\in\mathcal{F}_{f_i}$ such that $v_i\perp g_i$ and
$f_i=g_i+v_i$. Put $G_i:=\{t\in(a_i, b_i):\, g(t)>0\}$ and
$V_i=\{t\in(a_i, b_i):\, v(t)>0\}$. Since $f_{i}$ is strictly
positive on $(a_i,b_i)$ we deduce that $(a_i,b_i)=G_i\sqcup V_i$ and
therefore $G_i$ and $V_i$ are open-closed subsets of $(a_i,b_i)$.
But the interval $[a_i, b_i]$ is a connected set and we come to the
contradiction. Thus $f_{i}$ has no fragment $0<v_i<f_i$ and
therefore, either $f_i(t)=g_i(t)$ or $g_i(t)=0$ for any $t\in(a_i,
b_i)$. With each $g\in\mathcal{F}_{f}$ there is associated the sequence
$(g_i)_{i\in\Bbb{N}}$, where
$$
g_i=\begin{cases} 1,\quad \text{if} \; g(t)=f(t) \\
0,\quad \text{if} \; g(t)=0  \\
\end{cases} \quad t\in(a_i, b_i).
$$
Clearly, in this way a one-to-one correspondence between $\mathcal{F}_{f}$
and the set of all $0$\,-\,$1$ sequences is established.
Let $D$ be any fixed subset of $\mathcal{F}_{f}$. Put
\begin{gather*}
D_{+}:=\{i\in\N:\,\exists\, g\in D\;\text{such
that}\;g_i=1\};\\
D_{-}:=\{i\in\N:\; \text{such that} \, g_i=1 \; \text{for}\; \forall\,  g\in D\}.
\end{gather*}
Consider the pair of sequences  $(u_i)_{i\in\Bbb{N}}$ and
$(v_i)_{i\in\Bbb{N}}$, where
$$
u_i=\begin{cases}
1,\quad \text{if} \; i\in D_{+}\\
0,\quad \text{otherwise}, \\
\end{cases}
$$
$$
v_i=\begin{cases}
1,\quad \text{if} \; i\in D_{-}\\
0,\quad \text{otherwise}.\\
\end{cases}
$$
With $(u_i)_{i\in\Bbb{N}}$ and $(v_i)_{i\in\Bbb{N}}$ there is associated
the pair $\{u,v\}$ of fragments of $f$. Clearly $u=\sup D$ and
$v=\inf D$.
\end{proof}
Since the vector lattice $C[0,1]$ is Archimedean but not Dedekind complete the previous proposition  
shows that the set of all $C$-complete vector lattices is a new subclass of vector lattices which 
strictly contains the class of all Dedekind complete vector lattices. 
We note that in general a $C$-complete vector lattice is not Archimedian.
\begin{example}
Let $E=\Bbb{R}^2$ equipped with the lexicographic order. That is, we consider
$E$ as a vector lattice with the following order\footnote{\, The
cone $E_+$ in this vector lattice consists of the open right half-space $\{(x_1,x_2)\colon x_1>0\}$ joint with the half-ray
$\{(x_1,x_2)\colon x_1=0,\, x_2\geq 0\}$.}
$(x_1,x_2)\geq(y_1,y_2)$, whenever either $x_1>y_1$ or else $x_1=y_1$ and $x_2\geq y_2$.
The vector lattice $E$ is not Archimedian. On the other hand
it is not hard to verify  that the Boolean algebra of all fragments $\mathcal{F}_{x}$ of an arbitrary  
element $x=(x_1,x_2)\in E_{+}$ contains only two elements:
$$
\mathcal{F}_{x}=\{(x_1,x_2),\,(0,0)\}.
$$
Hence $E$ is $C$-complete.
\end{example}

 We say that a set $D\subset E$ is {\it laterally bounded}, if
there exits $x\in E$ such that $D\subset\mathcal{F}_{x}$. We say
that a laterally bounded set $D$ has a lateral supremum (infimum) if
there exists $u\in E$ ($v\in E$) such that $u=\sup D$
($v=\inf D$) with respect to the  partial order $\sqsubseteq$ in $\mathcal{F}_{x}$.
Taking into account Proposition~\ref{OPR}  we deduce that a vector
lattice $E$ is $C$-complete if and only if every laterally bounded
subset of $E$ has the lateral supremum and infimum.

\begin{definition} \label{def:ddmjf0}
Let $E$ be a vector lattice and $X$ be a normed space. An
orthogonally additive operator $T:E\to X$ is called {\it
narrow}, if for any $x\in E$ and  $\varepsilon > 0$ there exists a
pair $x_1, x_2$ of mutually complemented fragments of $v$, such that
$\|Tx_1-Tx_2\|<\varepsilon$. In particular, if $X=\Bbb{R}$, we call $T$ a
narrow functional.
\end{definition}
Observe that the image of an atom under a narrow operator $T$ is
zero. Indeed. The only disjoint fragments of an atom $a$ are $0$ and
$a$. So, due to the narrowness of $T$, for any $\eps>0$ one has $\nl
Tu\nr <\eps$, what means $Tu=0$.
\\
This is the reason for supposing the vector lattice $E$ to be atomless
in the Theorems \ref{thm-2} and in most of the
propositions of the current section.
\begin{example}
Let $(\Omega,\Sigma,\mu)$ be a $\sigma$-finite  measure space.
Consider a map $\mathcal{N}:L_{1}(\mu)\to \R$ defined by
$$
\mathcal{N}(f)=\| f\|_{L_{1}(\mu)},\quad f\in L_{1}(\mu).
$$
In \cite{PPW} (Proposition 2.5) it was shown that $\mathcal{N}$ is a narrow 
orthogonally additive functional on $L_{1}(\mu)$. 
\hide{
Indeed, since the norm in $L_1(\mu)$ is additive the equality
$$
\|f+g\|_{L_{1}(\mu)}=\nl
|f+g|\nr_{L_{1}(\mu)}=\||f|\|_{L_{1}(\mu)}+\||g|\|_{L_{1}(\mu)}=
\|f\|_{L_{1}(\mu)}+\|g\|_{L_{1}(\mu)}
$$
holds for any $f,g\in L_{1}(\mu)_+$ with $f\perp g$, we deduce that
$\mathcal{N}$ is an orthogonally additive functional. The narrowness
of the integral functional
$$
\mathcal{N}(f) = \|f\|_{L_{1}(\mu)}=\int_\Omega |f|\,d\mu,\quad f\in
L_{1}(\mu)
$$
is well known (see for instance \cite[Theorem~10.17]{PRan}).
}
\end{example}

A net $(x_\alpha)_{\alpha\in\Lambda}$ in a vector
lattice $E$ \textit{laterally converges} to $x \in E$ if $x_\alpha
\sqsubseteq x_\beta\sqsubseteq x$ for all $\alpha \leq\beta$
and $x_\alpha$ order converges to $x$. 
This is written as $x_\alpha \overset{\rm lat}\longrightarrow x$. 

\begin{definition}
An orthogonally
additive operator $T$ from a vector lattice $E$ to a normed space $X$ is called 
\textit{laterally-to-norm} continuous whenever for each laterally
convergent net $(x_\alpha)_{\alpha\in\Lambda}$ with $x_\alpha \overset{\rm
lat}\longrightarrow x$ the net $(Tx_\alpha)_{\alpha\in\Lambda}$
converges with respect to the norm to in $X$ to $Tx$.
\end{definition}

The following theorem is the second main result of the article.
\begin{thm} \label{thm-2}
Let $E$ be an atomless $C$-complete vector lattice and $X$ be
a Banach space. Then every orthogonally additive laterally-to-norm
continuous $C$-compact operator $T: E \to X$ is narrow.
\end{thm}

 The next auxiliary proposition is well known (see e.g.
\cite[Lemma~10.20]{PRan}).
\begin{prop} \label{prop-rounding}
Let  $(v_{i})_{i=1}^{n}$ be a finite subset of elements in a finite
dimensional normed space  $V$ and $(\lambda_{i})_{i=1}^{n}$ be non-negative numbers 
such that $0\leq\lambda_{i}\leq 1$ for each $i$.
Then there exists a set $(\theta_{i})_{i=1}^{n}$ of numbers such that 
$\theta_{i}\in\{0,1\},  \; i\in \{1,\ldots,n\}$ and 
$$
\Big\|\sum_{i=1}^{n}(\lambda_{i}-\theta_{i}) \, v_{i}\Big\|\leq\frac{\text{\rm
dim} \, V}{2}\max\limits_{i\in \{1,\ldots,n\}}\|v_{i}\|.
$$
\end{prop}

\begin{prop} \label{prop-dec}
Let be $E$ an atomless  vector lattice, $x\in E$, $X$ a Banach
space and  $T:E \to X$ an orthogonally additive laterally-to-norm
continuous operator. Then for any $\varepsilon > 0$ there exists a
decomposition $x=y\sqcup z$, where $y,z$ are nonzero fragments of
$x$ such that $\|Tz\|<\varepsilon$.
\end{prop}
\begin{proof}
Since $E$  is an atomless vector lattice the Boolean algebra
$\mathfrak{A}:=\mathcal{F}_{x}$ has infinite cardinality.
We note that the
set of all fragments of $x$ is the net $(x_\alpha)_{\alpha\in\mathfrak{A}}$
laterally converges to $x$.
The
laterally-to-norm continuity of $T$ implies the existence of
$\alpha_0\in\mathfrak{A}$ such that $\|Tx-Tx_{\alpha}\|<\varepsilon$
for all $\alpha\geq\alpha_{0}$. Consider for $\alpha_0$ the disjoint decomposition
$x=(x-x_{\alpha_0})\sqcup x_{\alpha_0}$. Then
\[Tx=T\big((x-x_{\alpha_0})\sqcup
x_{\alpha_0}\big)=T(x-x_{\alpha_0})+Tx_{\alpha}\quad\text{implies}\quad
\|T(x-x_{\alpha_0})\|<\varepsilon.
\]
 Assign $y:=x_{\alpha_0}$ and $z=x-x_{\alpha_0}$. Then $\|Tz\|<\varepsilon$
 and $x=y\sqcup z$ is the desirable disjoint decomposition.
\end{proof}

\begin{prop} \label{prop-decreas}
Let be $E$ an atomless vector lattice, $x\in E$, $X$ a Banach space and   
$T\colon E \to X$ an orthogonally additive laterally-to-norm continuous
operator.  
Assume that $(y_n)_{n\in \N}$ is a sequence of fragments of $x$ such that 
$y_1=x, \; y_n\sqsubseteq y_m$ for $n\geq m$; \,$m,n\in\N$ and
$\bigcap\limits_{n\in\N}\mathcal{F}_{y_n}=\{0\}$.  
Then $\lim\limits_{n\to\infty} \|T y_n\| = 0$.
\end{prop}
\begin{proof}
By definition $y_m=y_n\sqcup(y_m-y_n)$ for  $m,n\in\N$ with  $n\geq m$.
For $x_n:=x-y_n$, $n\in\N$ with $y_n=y_m - (y_m-y_n)$ we can write
\begin{gather*}
x_{n}=x-y_m+(y_m-y_n)= x-y_m+((x-y_n)-(x-y_m))=x_{m}+(x_n-x_m).
\end{gather*}
Clearly $x_{m}\perp(x_n-x_m)$ for $n,m\in\Bbb{N}$ and $n\geq m$ and
therefore $x_n\perp(x-x_n)$ for all $n\in \N$. Thus  $x=x_n \sqcup
(x-x_n),\,n\in\Bbb{N}$ and $\bigcap\limits_{n\in\N}\mathcal{F}_{x-x_n} =\bigcap\limits_{n\in\N}\mathcal{F}_{y_n}=\{0\}$.
Hence we deduce that the sequence $(x_n)_{n\in \N}$ laterally
 converges\footnote{\, Indeed, if some vector $0\neq u$ is a fragment of each
vector $x-x_n$ then $u\sqsubseteq(x-x_n)$
 and $|x-x_n|\leq u_n$ for some sequence $u_n\downarrow 0$ is impossible.} to
$x$.
Now the relation $x_{n}\perp(x-x_n)\,\,\text{for all}\,\,n\in\Bbb{N}$ implies
\begin{gather*}
Tx=T\big(x_n\sqcup(x-x_{n})\big)=
Tx_n+T(x-x_n)\,\,\text{and}\,\,T(x-x_n)=Tx-Tx_n.
\end{gather*}
 Thus by laterally-to-norm continuity of $T$ we have
that $Tx_n$ is norm convergent to $Tx$ in X and
$$
\lim\limits_{n\to\infty} \|Tx-Tx_n\|=\lim\limits_{n\to\infty} \|T y_n\|=0.
$$
\end{proof}

\begin{prop} \label{prop-dec-1}
Let be $E$  an atomless  $C$-complete vector lattice, $X$ a Banach
space, $T: E \to X$ an orthogonally additive laterally-to-norm
continuous operator, $x\in E$ and $\varepsilon > 0$. Then for some
$n\in \N$ there exists a decomposition
$x=\bigsqcup\limits_{i=1}^{n}x_i$, where $x_i$
 are nonzero fragments of $x$
such that  $\|Tx_{i}\|\leq\varepsilon$  for any $i\in
\{1,\ldots,n\}$.
\end{prop}
\begin{proof}
By Proposition~\ref{prop-dec} the set
$$
D_{x,T,\varepsilon}:=\{z\in\mathcal{F}_{x}:\,z\neq
0,\,\|Tz\|\leq\varepsilon\}
$$
is not empty. We note that $D_{x,T,\varepsilon}$ is a partially
ordered set with respect to the relation $\sqsubseteq$. Let
$(u_{\lambda})_{\lambda\in\Lambda} \subseteq D_{x,T,\varepsilon}$ be
a chain, where $\Lambda$ is a some linearly ordered index set.
Clearly $u_{\mu}\sqsubseteq u_\lambda$ for all
$\mu,\lambda\in\Lambda$, $\mu\sqsubseteq\lambda$.
By the $C$-completeness of vector lattice $E$ there
exists $u=\olim\limits_{\lambda\in\Lambda}(u_{\lambda})$, where
$u\in\mathcal{F}_{x}$. Due to $\|Tu_{\lambda}\|\leq \eps$ and the
laterally-to-norm continuity of $T$ we have that then
$Tu=\lim\limits_{\lambda\in\Lambda}(Tu_{\lambda})$ and therefore
$u\in D_{x,T,\eps}$. By Zorn's Lemma, there is a maximal
element\footnote{\, The element $z$ is maximal in $D_{x,T,\eps}$, if
$\nexists \,u\in D_{x,T,\eps}$ such that $z\sqsubseteq u$ and $z\neq
u$.} $z\in D_{x,T,\eps}$ with $\|Tz\|\leq\eps$. Put $y=x-z$. If
$\|Ty\|\leq\varepsilon$ then we got the required decomposition of
$x$. Otherwise we apply Proposition~\ref{prop-dec} to $y$ and get
$y=y_1\sqcup y_2$, where $y_1$ is a maximal element in
$D_{y,T,\eps}$ with $\|Ty_1\|\leq\eps$. In case of necessity, i.e.
if $\|Ty_2\|>\varepsilon$, by further continuing in the same way
with the corresponding fragments
$y_2, y_4, \ldots, y_{2k}, \ldots$ of $y$  we construct a sequence of decompositions
$y_{2k}=y_{2k+1}\sqcup y_{2k+2}$, where $y_{2k+1}$ is a maximal element
in $D_{y_{2k},T,\eps}$ and satisfying the conditions
$\|Ty_{2k+1}\|\leq\varepsilon$ and $\|Ty_{2k+2}\| > \varepsilon, \;
k\in \N$. We claim that there exists $l\in\N$ such that
$y_{2l}=y_{2l+1}\sqcup y_{2l+2}$ and both $\|Ty_{2l+1}\|, \,
\|Ty_{2l+2}\|\leq\varepsilon$.
Assume the contrary. Then the sequence
$(y_{2k})_{k\in \N}$
of fragments of $y$ is
such that $\|Ty_{2k}\| > \varepsilon$ for all $k\in\N$.
Nevertheless we show
$\bigcap\limits_{k\in\N}\mathcal{F}_{y_{2k}}=\{0\}$.
Indeed, assume that there exists a nonzero element
$v\in\bigcap\limits_{k\in\N}\mathcal{F}_{y_{2k}}$. Then for the
sequence
$$
y'_{2}=y_{2}-v,\; y'_{4}=y_{4}-v, \;\ldots,\; y'_{2k}=y_{2k}-v,\;\ldots
$$
we have $y'_{2n}\sqsubseteq y'_{2m}$ for $m,n\in\N$ with $n\geq m$ and
$\bigcap\limits_{n\in\N}\mathcal{F}_{y'_{2n}}=\{0\}$.
According to Proposition~\ref{prop-decreas} there exists $n_{0}\in\N$ such that
$\|Ty'_{2n_0}\| < \varepsilon$.
Thus $y_{2n_0}=y'_{2n_0}\sqcup v$,
$y'_{2n_0}\in D_{y_{2n_0},T,\eps}$ and
$y_{{2n_0}+1}$ is a maximal element of $D_{y_{2n_0},T,\eps}$. We have
\[y_{2n_0}=y_{{2n_0}+1}\sqcup y_{{2n_0}+2}= y_{2n_0}'\sqcup v. \]

Consider now the two cases:
\[
\mathcal{F}_{y_{_{{2n_0}+1}}} \cap\mathcal{F}_{v}=\{0\}\quad \text{and}\quad \mathcal{F}_{y_{_{{2n_0}+1}}} \cap\mathcal{F}_{v}\neq \{0\}.
\]
In the first case, since no non-zero fragment of $v$ is a fragment of $y_{{2n_0}+1}$,
we get $y_{{2n_0}+1}\sqsubseteq y'_{2n_0}$.
Observe that - due to the maximality of
 $y_{{2n_0}+1}$ in $D_{y_{2n_0},T,\eps}$ -
the relation $y_{{2n_0}+1}\sqsubseteq y'_{2n_0}$, $y_{{2n_0}+1}\neq y'_{2n_0}$ is impossible.
Hence  $y'_{2n_0}= y_{2n_0+1}$, i.e. $y_{2n_0+2}=v$ and, therefore $v$ is
not a fragment of $y_{2n_0+4}$.
In  the second one there is a nonzero
fragment $w\in\mathcal{F}_{y_{_{{2n_0}+1}}} \cap\mathcal{F}_{v}$.
Thus $v=(v-w)\sqcup w$ and $w\sqsubseteq y_{{2n_0}+1}$,
i.e. the  non-zero fragment $w$ of $v$ belongs to $\mathcal{F}_{y_{2n_0+1}}$.
Therefore the element $v$ can not be a fragment of $y_{2n_0+2}$.

Hence $\bigcap\limits_{k\in\N}\mathcal{F}_{y_{2k}}=\{0\}$ and so, again by
applying  Proposition~\ref{prop-decreas}, we get
$\lim\limits_{k\to\infty} \|T y_{2k}\|=0$. This is a contradiction
and therefore the desirable $l\in\N$ exists. Consider the following elements
$$
x_{1}=y_1, \;x_2= y_3,\; \ldots, \; x_l=y_{2l-1},\; x_{l+1}= y_{2l+1},
\;x_{l+2}= y_{2l+2}, \:x_{l+3}= z.
$$
Then  $x=\bigsqcup\limits_{i=1}^{n}x_{i}$ with $n=l+3$ is the
desirable  decomposition of $x$.
\end{proof}
\begin{prop} \label{le-fin}
Let $E$ be an atomless  $C$-complete vector lattice and $V$ a finite
dimensional Banach space. Then every orthogonally additive
laterally-to-norm continuous operator $G: E \to V$ is
narrow.
\end{prop}
\begin{proof}
Fix any $x\in E$ and $\varepsilon > 0$.
According to Proposition~\ref{prop-dec-1} there is a disjoint
decomposition
$x=\bigsqcup\limits_{i=1}^{n}x_i$ such that
$\|Gx_{i}\| < \frac{\eps}{{\rm dim} \, V}$  for any
$i\in \{1,\ldots,n\}$.
 Then by using Proposition~\ref{prop-rounding} for $\la_i=\frac{1}{2}$
there exist numbers
 $\theta_i\in\{0,1\}$ for $i\in \{1,\ldots, n\}$ such that
\begin{gather} \label{eq:855jfg}
2\Bigl\|\sum_{i=1}^{n} \Bigl(\frac{1}{2}-\theta_{i} \Bigr) \,
Gx_i\Bigr\| \leq {\rm dim}\,V \max\limits_{i\in \{1,\ldots,n\}}
\|Gx_i\| < \eps.
\end{gather}
Observe that for $I_0 = \big\{i\in \{1,\ldots,n\big\}\colon \theta_i
= 0\big\}$ and $I_1 = \big\{i\in \{1,\ldots, n\}\colon \theta_i =
1\big\}$ the vectors $y_k = \bigsqcup\limits_{i \in I_k} x_i$ for $k\in
\{0,1\}$ are mutually complemented fragments of $x$ and by
\eqref{eq:855jfg},
$$
\|Gy_1 - Gy_2\| = \Bigl\|\sum_{i\in I_0\sqcup I_1} (1-2\theta_i) \,
Gx_i\Bigr\| < \eps,
$$
hence the operator $G$ is narrow.
\end{proof}

\begin{definition}
Let $E$ be a  vector lattice and $F$ a vector space. An orthogonally additive
operator $T:E\to F$ is called of {\it finite rank} if
the set $T(E)$ generates a finite-dimensional subspace in $F$.
\end{definition}

Now we are in the position to prove the second main result (Theorem \ref{thm-2}).
\\
Before, however, we notice that a Banach space $X$ can be considered
as a closed subspace of the Banach space 
\[
W:=l_{\infty}(B_{X^*})=\{q\colon B_{X^*}\to \R,\, \sup|q(f)|<\infty,\; f\in B_{X^*} \}
\]
(equipped with the supremum-norm) of all real-valued bounded functions on the closed unit ball $B_{_{X^*}}$
of the dual space $X^*$ according to 
$$
X \hookrightarrow X^{**} \hookrightarrow W,
$$
where the notation $\hookrightarrow$ means the isometric embedding
\[X\ni x\mapsto F_x\in X^{**} \quad \mbox{given by}\quad  F_x(f):=f(x),\; f\in
B_{_{X^*}}.
\]
Observe that, if $H$ is a relatively compact subset of the Banach space $W$
and $\eps > 0$, then there exists a linear finite rank operator $R\in
\cL(W)$ such\footnote{\, By $\cL(W)$ there is denoted the space of all linear bounded operators on $W$.} 
that $\|w-Rw\|\leq\eps$ for every $w\in H$
\cite[Lemma~10.25]{PRan}.

\begin{proof}[Proof of Theorem~\ref{thm-2}]
Fix an arbitrary $x\in E$ and $\eps > 0$. Due to the $C$-compactness of $T$ 
the set $K = T(\mathcal{F}_{x})$ is relatively compact in $X$ and therefore in $W$.
It follows that there exists a finite rank operator $S\in\cL(W)$ such that $\|y-
Sy\|\leq\frac{\varepsilon}{4}$ for every $y\in K$. 
Then $G=S\circ T$
is an orthogonally additive laterally-to-norm continuous
finite rank operator. By Lemma~\ref{le-fin} $G$ is narrow and 
consequently, there exist mutually complemented fragments
$x_{1},x_{2}$ of $x$ such that $\|Gx_1 - Gx_2\| < \frac{\eps}{2}$.
Therefore,
\begin{align*}
& \|Tx_1 - Tx_2 \| =\\
& \|Tx_1 - Tx_2 + S(Tx_1) - S(Tx_2) - S(Tx_1) + S(Tx_2) \|=   \\
& \|Tx_1 - Tx_2 + Gx_1 - Gx_2 - Gx_1+Gx_2\| \leq     \\
& \|Gx_1 - Gx_2\| + \|Tx_1 - Gx_1\|+\|Tx_2 - Gx_2\| <
\frac{\varepsilon}{2}+\frac{\varepsilon}{2}=\varepsilon
\end{align*}
since $\|Tx_i - Gx_i\|<\frac{\varepsilon}{4}$ for $i\in \{1,2\}$.
This completes the proof.
\end{proof}
We remark that Theorem~\ref{thm-2} generalizes the result of the article
\cite[Theorem~3.2]{PP}.

\medskip
{\bf Acknowledgments.} Marat Pliev  was supported  by the Russian Foundation
for Basic Research (grant  number 17-51-12064). Martin Weber was supported by the
Deutsche Forschungsgemeinschaft (grant number CH 1285/5-1, Order preserving
operators in problems of optimal control and in the theory of partial differential equations).

\end{document}